\documentclass[a4paper,11pt,reqno] {amsart}
\reversemarginpar  
\usepackage{xcolor}
\usepackage{eucal}
\usepackage{amsthm}
\usepackage{amsfonts}
\usepackage{amsmath}
\usepackage{amssymb}
\usepackage{mathrsfs}
\usepackage{epsfig}

 \usepackage{float}

\newcommand{\ncom}{\newcommand}
\ncom{\ul}{\underline}
\ncom{\ol}{\overline}
\ncom{\bq}{\begin{equation}}
\ncom{\eq}{\end{equation}}
\ncom{\beqn}{\begin{eqnarray*}}
\ncom{\eeqn}{\end{eqnarray*}}
\ncom{\beq}{\begin{eqnarray}}
\ncom{\eeq}{\end{eqnarray}}
\ncom{\nno}{\nonumber}
\ncom{\rar}{\rightarrow}
\ncom{\Rar}{\Rightarrow}
\ncom{\noin}{\noindent}
\ncom{\bc}{\begin{centre}}
\ncom{\ec}{\end{centre}}
\ncom{\sz}{\scriptsize}
\ncom{\rf}{\ref}
\ncom{\sgm}{\sigma}
\ncom{\Sgm}{\Sigma}
\ncom{\dt}{\delta}
\ncom{\Dt}{Delta}

\ncom{\s}{\underline{s}}
\ncom{\lmd}{\lambda}
\ncom{\Lmd}{\Lambda}
\ncom{\eps}{\epsilon}
\ncom{\pcc}{\stackrel{P}{>}}
\ncom{\dist}{{\rm\,dist}}
\ncom{\sspan}{{\rm\,span}}
\ncom{\re}{{\rm Re\,}}
\ncom{\im}{{\rm Im\,}}
\ncom{\sgn}{{\rm sgn\,}}
\ncom{\ba}{\begin{array}}
\ncom{\ea}{\end{array}}
\ncom{\eop}{\hfill{{\rule{2.5mm}{2.5mm}}}}
\ncom{\eoe}{\hfill{{\rule{1.5mm}{1.5mm}}}}
\ncom{\eof}{\hfill{{\rule{1.5mm}{1.5mm}}}}
\ncom{\hone}{\mbox{\hspace{1em}}}
\ncom{\htwo}{\mbox{\hspace{2em}}}
\ncom{\hthree}{\mbox{\hspace{3em}}}
\ncom{\hfour}{\mbox{\hspace{4em}}}
\ncom{\hsev}{\mbox{\hspace{7em}}}
\ncom{\vone}{\vskip 2ex}
\ncom{\cH}{{\mathcal H}}
\ncom{\vtwo}{\vskip 4ex}
\ncom{\vonee}{\vskip 1.5ex}
\ncom{\vthree}{\vskip 6ex}
\ncom{\vfour}{\vspace*{8ex}}
\ncom{\norm}{\|\;\;\|}
\ncom{\integ}[4]{\int_{#1}^{#2}\,{#3}\,d{#4}}
\ncom{\inp}[2]{\langle{#1},\,{#2} \rangle}
\ncom{\Inp}[2]{\big\langle{#1},\,{#2} \big\rangle}
\ncom{\vspan}[1]{{{\rm\,span}\#1 \}}}
\ncom{\dm}[1]{\displaystyle {#1}}
\ncom{\Hom}{\operatorname{Hom}}
\ncom{\Hol}{\operatorname{Hol}}
\ncom{\Ps}{\mathcal P_{\underline{s}}}
\ncom{\hl}{\mathcal H}

\ncom{\defin} {\overset {\text {\rm def} }{=}}

\newtheorem{theorem}{\bf Theorem}[section]
\newtheorem{proposition}[theorem]{\bf Proposition}
\newtheorem{corollary}[theorem]{\bf Corollary}
\newtheorem{lemma}[theorem]{\bf Lemma}

\newtheorem{conjecture}[theorem]{\bf Conjecture}

\newtheorem{remark}[theorem]{\bf Remark}
\newtheorem{definition}[theorem]{\bf Definition}

\def \s{\underline{s}}
\renewcommand{\epsilon}{\varepsilon}
\renewcommand{\kappa}{\varkappa}

\begin{document}

\title[$\mathbb{K}$-homogeneous tuple]{$\mathbb{K}$-homogeneous tuple of operators on bounded symmetric domains}

\author[S. Ghara]{Soumitra Ghara}
\address[]{Department of Mathematics and Statistics, Indian Institute of Technology,
Kanpur 208016, India} \email{ghara90@gmail.com}
\author[S. Kumar]{Surjit Kumar}
\author[P. Pramanick]{Paramita Pramanick}
\address[]{Department of Mathematics, Indian Institute of Science,
Bangalore 560012, India}
 \email{surjitkumar@iisc.ac.in}
\email{paramitap@iisc.ac.in}
 \thanks{The resarch of the first author was supported by National Post Doctoral Fellowship (Ref. No. PDF/2019/002724).\\
 The work of the second author was supported by Inspire Faculty Fellowship (Ref. No. DST/INSPIRE/04/2016/001008) and the third author was supported by National Board for Higher Mathematics (Ref. No. 2/39(2)/2015/NBHM/R \& D II/7416).}

\subjclass[2010]{Primary 47A13, 47B32, 46E20, Secondary 32M15}
\keywords{spherical tuples, bounded symmetric domain, weighted Bergman spaces, Cowen-Douglas class operators, homogeneous operators}

\date{}

\begin{abstract}
	Let $\Omega$ be an irreducible bounded symmetric domain of rank $r$ in $\mathbb C^d.$ Let $\mathbb K$ be the maximal compact subgroup of the identity component $G$ of the biholomorphic automorphism group  of the domain $\Omega$.  The group $\mathbb K$ consisting of linear transformations acts naturally  on any $d$-tuple $\boldsymbol T=(T_1,\ldots, T_d)$ of commuting bounded linear operators. 
If the orbit of this action modulo unitary equivalence is a singleton, then we say that $\boldsymbol T$ is $\mathbb{K}$-homogeneous. In this paper, we obtain a model for all $\mathbb{K}$-homogeneous  $d$-tuple $\boldsymbol{T}$  as the operators of multiplication by the coordinate functions $z_1,\ldots ,z_d$ on a reproducing kernel Hilbert space of holomorphic functions defined on $\Omega$. Using this model we obtain a criterion for (i) boundedness, (ii) membership in the Cowen-Douglas class (iii) unitary equivalence and similarity of these $d$-tuples.  In particular, we show that the adjoint of the $d$-tuple of multiplication by the coordinate functions on the weighted Bergman spaces are in the Cowen-Douglas class $B_1(\Omega)$.  For a bounded symmetric domain $\Omega$ of rank $2$, an explicit description of the operator  $\sum_{i=1}^d T_i^*T_i$ 
is given. In general, based on this formula, we make a conjecture giving the form of this operator. 
\end{abstract}

\renewcommand{\thefootnote}{}
%
%
%
%

\maketitle




\section{Introduction}
It was noted in \cite[Corollary 2]{Sh} that a weighted shift operator  $T$ is circular in the sense that $T$ is unitarily equivalent to $c T$ whenever $|c| = 1$. This class of operators was studied further by several authors \cite{G}, \cite{AHDH}.
In \cite{CY}, Chavan and Yakubovich generalized this notion to spherical tuple of operators.
A $d$-tuple $\boldsymbol T=(T_1,\ldots,T_d)$  of commuting operators is said to be 
{\it spherical} if $U\cdot \boldsymbol T$ is unitarily equivalent to $\boldsymbol T$ for all  unitary matrix $U$ in the group  $\mathcal U(d)$ of $d\times d$ unitary matrices. Here 
$U\cdot \boldsymbol T$ is the natural action of $\mathcal U(d)$ on the $d$-tuple $\boldsymbol T$. Chavan and Yakubovich proved that under some mild hypothesis, every spherical $d$-tuple is unitarily equivalent to the $d$-tuple $\boldsymbol{M}=(M_1,\ldots,M_d)$ of multiplication operators by the coordinate function $z_1,\ldots,z_d$ on  a reproducing kernel Hilbert space determined by a $\mathcal U(d)$-invariant kernel function $\sum_{j=0}^{\infty} a_j \inp{z}{w}^j$ defined on the open Euclidean unit ball $\mathbb B^d$ in $\mathbb C^d$.
One of our main objectives in this paper is to explore a notion analogous to that  of spherical operator tuples in the context of a bounded symmetric domain.

Bounded symmetric domains are the natural generalization of open unit disc in one complex variable and open Euclidean unit ball in several complex variables.
A bounded domain $\Omega \subset \mathbb{C}^d$ is said to be {\it symmetric} if for every $z \in \Omega,$ there exists a biholomorphic automorphism on $\Omega$ of period two,  having $z$ as an isolated fixed point. The domain $\Omega$ is said to be  {\it irreducible}  if it is not biholomorphically isomorphic to  a product of two non-trivial domains. We refer to \cite{Loos}, \cite{Arazy} for the definition and basic properties of bounded symmetric domains.

Let $\Omega$ be an irreducible bounded symmetric domain in $\mathbb C^d$  and let
$ \rm{Aut}(\Omega)$ denote the group of  biholomorphic automorphisms of $\Omega$, equipped with the topology of uniform convergence on compact subsets of $\Omega$. Let $G$ denote the connected component of identity in $\rm{Aut}(\Omega)$. It is known that $G$ acts transitively on $\Omega$. Let $\mathbb{K}$ be the subgroup of linear automorphisms in $G$. By Cartan's theorem \cite[Proposition 2, pp. 67]{RN}, $\mathbb{K}=\{\phi\in G:\phi(0)=0\}$. $\mathbb{K}$ is known to be a maximal compact subgroup of $G$ and $\Omega$ is ismorphic to $G/{\mathbb{K}}$. 
Note that $\mathcal U(d)$ is the subgroup of linear biholomorphic automorphisms of $\rm Aut(\mathbb B^d)$. Therefore, it is natural to replace $\mathcal U(d)$ with the subgroup $\mathbb{K}$ of linear biholomorphic automorphisms of an irreducible bounded symmetric domain  $\Omega$ and study all operator $d$-tuples $\boldsymbol T$ such that $k \cdot \boldsymbol T$ is unitarily equivalent to $\boldsymbol T$ for all $k \in \mathbb{K}$. The action of the group $\mathbb K$ on the $d$-tuples is defined below.
 The group $\mathbb{K}$ acts on $\Omega$ by the rule 
 \[k\cdot \boldsymbol z:=\big(k_1(\boldsymbol z), \ldots, k_d(\boldsymbol z)\big),\qquad k\in \mathbb{K} \mbox{~and~} \boldsymbol z\in \Omega.\]
 Note that $k_1(\boldsymbol z), \ldots, k_d(\boldsymbol z)$ are linear polynomials. Thus $k\in \mathbb{K}$ acts on any commuting $d$-tuple of bounded linear operators $\boldsymbol{T}=(T_1, \ldots, T_d)$, defined on complex seperable Hilbert space $\mathcal{H}$, naturally, via the map \[k\cdot\boldsymbol{T}:=\big(k_1(T_1, \ldots, T_d), \ldots, k_d(T_1, \ldots, T_d)\big). \] 
 \begin{definition}
 	A $d$-tuple of commuting bounded linear operators $\boldsymbol{T}=(T_1,T_2,\ldots,T_d)$  on $\mathcal{H}$ is said to be $\mathbb{K}$-homogeneous if for all $k$ in $\mathbb{K}$ the operators $\boldsymbol{T}$ and $k\cdot \boldsymbol{T}$ are unitarily equivalent, that is, for all $k$ in $\mathbb{K}$ there exists a unitary operator $\Gamma(k)$  on $\mathcal{H}$ such that 
 	\[T_j\Gamma(k)=\Gamma(k) k_j(T_1, \ldots, T_d),\qquad j=1,2,\ldots,d.\]
 	For brevity, we will write \[\boldsymbol {T} \Gamma(k)= \Gamma(k) (k\cdot\boldsymbol{T}).\]
 \end{definition}

We point out that the commuting operator tuples $\boldsymbol T =(T_1,\ldots, T_d)$ such that $\boldsymbol T$ and $g(\boldsymbol T)$ are unitarily equivalent for all $g$ in $G$, called {\it homogeneous} tuples,  have been also studied over the past few years, see \cite{MS}, \cite{MU}, \cite{MK1}.
Cowen and Douglas introduced a class of operators $B_n(\Omega)$ in the very influential paper \cite{CD}, where $\Omega\subset \mathbb C$ is a bounded domain and $n$ is a positive integer. 
In the case of open unit disc $\mathbb D$, all homogeneous operators in $B_1(\mathbb D)$ were classified by Misra in \cite{Misra}. As a corollary of his abstract classification theorem, Wilkins  provided an explicit model for all homogeneous operators in $B_2(\mathbb D)$, see  \cite{Wilkins}. Later in 2011, using techniques from complex geometry and representation theory, a complete classification of homogeneous operators in  the Cowen-Douglas class $B_n(\mathbb D)$ was obtained by Misra and Kor\'anyi in \cite{MK}. 
Homogeneous operators on an irreducible  bounded symmetric domain of type $I$, see below,  were studied by Misra and Bagachi in \cite{MisraBagchi}.
Later in  \cite{Zhang}, their results were generalized for an arbitrary irreducible bounded symmetric domain by Arazy and Zhang.


First, we fix some notations, which will be used throughout this paper. It will also help us describe our results.
For a set $X$ and positive integer $m$,  $X^m$ stands for the $m$-fold Cartesian product of $X$.
The symbols ${\mathbb N_0}, \mathbb R$ and $\mathbb C$ stand for the set of non-negative
integers, field of the real numbers and the field of complex numbers, respectively. Let $\mathcal H$ be a complex Hilbert space. Let $\mathcal{B}({\mathcal H})$ denote the unital Banach algebra of bounded linear operators on $\mathcal H.$ If $T \in \mathcal B(\mathcal H)$, then $\ker T$ denotes the kernel of $T$, the range of $T$ is denoted by $\text{ran}T$.

Every irreducible bounded symmetric domain $\Omega$ of rank $r$ can be realized as an open unit ball of a {\it Cartan factor} $Z = \mathbb C^d$. For a fixed frame $e_1, \cdots , e_r $ of pairwise orthogonal minimal tripotents, let
\[ Z= \sum_{0\leq i \leq j \leq r}Z_{ij}\]
be the  {\it joint Peirce decomposition} of $Z$ (see \cite[ pp. 57]{upmeier}). Note that $Z_{00}=\{0\}$ and $Z_{ii}=\mathbb C e_i$ for all $i=1, \ldots,r.$ Moreover, \[a:=\dim Z_{ij}, \qquad 1\leq i<j \leq r \] is independent of $i,j$ and 
\[b:=\dim Z_{0j}, \qquad 1\leq j \leq r \] is independent of $j$. The parameters $a,b$ are known to be the {\it characterstic multiplicities} of $Z$ and the numerical invariants $(r,a,b)$ determine the domain  $\Omega $ uniquely upto biholomorphic equivalence (see \cite{Arazy}). The dimension $d$ is related to the  numerical invariants $(r,a,b)$ as follows:
\[d=r+\frac{a}{2}r(r-1)+rb.\]

According to the classification due to  E. Cartan \cite{Ca}, there are six types of irreducible bounded symmetric domains upto biholomorphic equivalence (see also  \cite{Loos}). The first four types of these domains are called the classical Cartan domains, while other two types are known as the exceptional domains.
In what follows, we consider only the classical domains, that is, an irreducible bounded symmetric domain of one of the following four types: 
\begin{enumerate}
	\item [(i)]{\bf Type I} : $n\times m \; (m\geq n)$  complex matrices $\boldsymbol  z$ with $\|\boldsymbol z\| < 1.$  These  domains are determined by the triple $ (n,2,m-n). $
	\item[(ii)] {\bf Type II} : symmetric complex matrices $\boldsymbol z$ of order $n$ with $\|\boldsymbol z\| < 1.$ In this case, the triple $ (n, 1, 0)$  is complete biholomorphic invariant. 
	\item [(iii)]{\bf Type III} : $n \times n $ anti-symmetric complex matrices $\boldsymbol z$ of order $n$ with $\|\boldsymbol z\| < 1.$  Here  $r= \left[\frac{n}{2} \right] , a= 4$ and $b= 0$ if $n$ is even and $b= 2$  if $n$ is odd.
	\item [(iv)]{\bf Type IV}  ({\it The Lie ball}) : all $\boldsymbol z \in \mathbb C^d\; (d \geq 5)$ such that $1+|\tfrac{1}{2}\boldsymbol {z}^{t}\boldsymbol {z}|^2 > \overline {\boldsymbol z}^{t} \boldsymbol {z} $ and $\overline {\boldsymbol z}^{t} \boldsymbol {z}< 2,$ where $\overline {\boldsymbol z}^{t} $ is the complex conjugate of the transpose $\boldsymbol {z}^{t}.$  The triple $(2, d-2, 0)$ is complete biholomorphic invariant for these domains.
\end{enumerate}

Let $\mathcal{P}$ be the space of all analytic polynomials on $Z.$ For all  $n \in \mathbb N_0,$ let $\mathcal{P}_n$ denote the subspace of $\mathcal{P}$ consisting of all homogeneous polynomials of degree $n$. 
Clearly, as a vector space, $\mathcal{P}$ can be written as the direct sum $\sum_{n=0}^{\infty}\mathcal{P}_n$.
The group $\mathbb K$ acts on the space  $\mathcal P$ by composition, that is, $(k.p)(\boldsymbol z)=p(k^{-1} \boldsymbol z)$, $k\in \mathbb K,~ p\in \mathcal P.$ Below we describe the irreducible components of this action.
An r-tuple $\s=(s_1,s_2,\ldots,s_r)$ is called a {\it signature} if $s_1\geq s_2\geq \ldots \geq s_r \geq 0.$ Let $\overrightarrow  {\mathbb N}^r_0$ denote the set of all signatures. We associate the {\it conical polynomial}, see \cite[pp. 128]{upmeier} for the definition,  
$$\Delta_{\s}(\boldsymbol z)=\Delta_{1}^{s_1-s_2}(\boldsymbol z)\ldots \Delta_{r-1}^{s_{r-1}-s_r}(\boldsymbol z)\Delta_{r}^{s_r}(\boldsymbol z)$$  for all $\s \in \overrightarrow  {\mathbb N}^r_0$ 
and the polynomial space $\mathcal{P}_{\s}=\vee \{\Delta_{\s}\circ k;k\in \mathbb{K}\}$. It is known that  the polynomial space $\{ \mathcal{P}_{\s}\}_{\s \in \overrightarrow  {\mathbb N}^r_0}$ are precisely the $\mathbb{K}$-invariant, irreducible subspaces of $\mathcal{P}$ which are mutually $\mathbb{K}$-inequivalent, and
$$\mathcal{P}=\sum_{\s \in \overrightarrow  {\mathbb N}^r_0}\mathcal{P}_{\s }.$$
The Fischer-Fock inner product on $\mathcal{P}$, defined by $\langle p,q\rangle_{F}:=\frac{1}{\pi^d}\int_{\mathbb{C}^d}p(\boldsymbol z)\overline{q(\boldsymbol z)}e^{-|\boldsymbol z|^2}dm(\boldsymbol z)$, is $\mathbb{K}$-invariant. The reproducing kernel of the space $\mathcal{P}_{\s }$ with respect to the Fischer-Fock inner product is denoted by $K_{\s}(\boldsymbol z, \boldsymbol  w).$ Note that
\[\sum_{\s \in \overrightarrow  {\mathbb N}^r_0} {K}_{\s }(\boldsymbol z,\boldsymbol w) = e^{\boldsymbol z\cdot \overline{\boldsymbol w}}.\]
Further any $\mathbb{K}$-invariant Hilbert space $\mathcal{H}$ of analytic functions on $\Omega$ has the decomposition $$\mathcal{H}=\oplus_{\s\in \overrightarrow  {\mathbb N}^r_0} \mathcal{P}_{\s}.$$ This decomposition is called {\it Peter-Weyl decomposition} \cite{Hu}. 

Let $\boldsymbol T =(T_1, \ldots, T_d)$ be a commuting $d$-tuple of bounded linear operators acting on a complex separable Hilbert space $\mathcal H$. Also, let 
 $D_{\boldsymbol{T}}:\mathcal{H}\to\mathcal{H}\oplus\cdots \oplus\mathcal{H}$ be the operator 
\[D_{\boldsymbol{T}}h:=(T_{1}h,\ldots,T_{d}h), h\in \mathcal{H}.\]
We note that $\ker D_{T}=\cap_{i=1}^d \ker T_{i}$ is the {\it joint kernel} and 
$\sigma_p(\boldsymbol T)  = \{\boldsymbol w\in \mathbb C^d: \ker D_{\boldsymbol T -\boldsymbol w I}\not = \boldsymbol 0 \}$ is the {\it joint point spectrum} of the $d$-tuple $\boldsymbol{T}=(T_1,\ldots,T_d)$. Throughout this paper we will study a class of $\mathbb{K}$-homogeneous $d$-tuples, which is defined below. 
\begin{definition} A commuting $d$-tuple of $\mathbb{K}$-homogeneous operators $\boldsymbol T$ possessing the following properties  
\begin{enumerate}
\item[(i)] $\dim \ker D_{\boldsymbol T^*} = 1$ 
\item[(ii)] any non-zero vector $e$ in $\ker D_{\boldsymbol T^*}$ is cyclic for $\boldsymbol T$. 
\item[(iii)] $\Omega \subseteq \sigma_p(\boldsymbol T^*)$
\end{enumerate}
is said to be in the class $\mathcal A\mathbb{K}(\Omega)$. 
\end{definition}
In this paper, we provide a concrete model for all the commuting $d$-tuples $\boldsymbol T$ (which are necessarily of $\mathbb{K}$- homogeneous) in the class $\mathcal A\mathbb{K}(\Omega)$ as multiplication by the coordinate functions $z_1,\ldots ,z_d$ on a reproducing kernel Hilbert space of holomorphic functions $(\mathcal H, K)$ defined on $\Omega$. We describe the kernel $K$ in terms of the invariant kernels $K_{\s}$ of the spaces $\mathcal P_{\s}$. 

Having described the model, we obtain a criterion for boundedness of these operators. Using this criterion, we determine which $d$-tuple of multiplication operators on the {\it weighted Bergman spaces} are bounded. The boundedness criterion for  the multiplication operators on the weighted Bergman spaces has appeared before in \cite{MisraBagchi}, \cite{Zhang}. 

We also obtain criterion for the adjoint of the $d$-tuple of operators in $\mathcal A\mathbb{K}(\Omega)$ to be in the Cowen-Douglas class $B_1(\Omega_0)$ for some neighbourhood $\Omega_0\subset \Omega$ of $0\in \Omega$. In case of weighted Bergman spaces $\mathcal H^{(\nu)}$, we prove that the adjoint of the $d$-tuple of multiplication operators by the coordinate functions are in the Cowen-Douglas class $B_1(\Omega)$.

For any $\boldsymbol T$ in the class $\mathcal A\mathbb{K}(\Omega)$, we point out that  the operators $\sum_{i=1}^d T_i^*T_i$ and $\sum_{i=1}^d T_iT_i^*$ restricted to the subspace $\mathcal P_{\s}$ are  scalar times the identity.  In particular, for the  weighted Bergman spaces $\mathcal H^{(\nu)}$, \cite [Proposition 4.4]{Zhang}  provides an explicit form for the  operator $\sum_{i=1}^d T_iT_i^*$. We extend this formula for  any $\boldsymbol T$ in the class $\mathcal A\mathbb{K}(\Omega)$. Moreover, for the Hardy space of the Shilov boundary $S$ of $\Omega$, we show that $\sum_{i=1}^d M_i^*M_i$ is the rank times identity, see also \cite{At}.  Also, for any $\boldsymbol T$ in  $\mathcal A\mathbb{K}(\Omega)$, we have computed the operator $\sum_{i=1}^d T_i^*T_i$ on certain subspaces of $\hl$, and as a consequence, it is shown that the commutators $[M_i^*,M_i]$, $i=1,\ldots,d$, on the weighted Bergman spaces are compact  if and only if $r=1$. For any domain $\Omega$ of rank $2$, we obtained an explicit description of the operator  $\sum_{i=1}^d T_i^*T_i$ and conjectured the form of this operator for a domain of any rank $r>2$. This conjecture was proved by Upmeier, see \cite{Up}. 

Finally, we  study the question of unitary equivalence and similarity of $d$-tuples of operators in the class $\mathcal A\mathbb{K}(\Omega).$


\section{Model for operators in $\mathcal A\mathbb{K}(\Omega)$
}
We begin this section by providing a well known family of examples, namely, the $d$-tuple of multiplication by the coordinate functions on the weighted Bergman spaces, which belongs to the class $\mathcal A\mathbb{K}(\Omega)$.

For $\nu  \in \left \{0, \ldots, \frac{a}{2}(r-1)\right \}\cup \left( \frac{a}{2}(r-1), \infty\right),$  so called {\it Wallach set} of $\Omega$ (see \cite{Wallachset}), consider the weighted Bergman kernel 
\[ K^{(\nu)} (\boldsymbol z,\boldsymbol w)=\sum_{\s} (\nu)_{\s}K_{\s}(\boldsymbol z,\boldsymbol w) ,\qquad \boldsymbol z, \boldsymbol w \in \Omega,\] 
where $(\nu)_{\s}$ is the generalized Pochhammer symbol
\[(\nu)_{\s}:= \prod_{j=1}^{r} \left( \nu-\frac{a}{2}(j-1)\right)_{s_{j}}= \prod_{j=1}^{r}\prod_{l=1}^{s_{j}} \left(\nu-\frac{a}{2}(j-1)+l-1\right).\]
Let $\mathcal H^{(\nu)}$ denote the  weighted Bergman space of holomorphic functions on $\Omega$ determined by the reproducing kernel $K^{(\nu)}.$ If $\nu=\tfrac{d}{r}$ and $\nu= \tfrac{a}{2}(r-1)+\tfrac{d}{r},$ then the  weighted Bergman spaces $\mathcal H^{(\nu)}$ coincide with the Hardy space $H^2(S)$ over the {\it Shilov boundary} $S$ of $\Omega$
and the classical Bergman space $\mathbb A^2(\Omega)$ respectively. For $\nu > \tfrac{a}{2}(r-1),$ the multiplication $d$-tuple $\boldsymbol M^{(\nu)}= (M^{(\nu)}_1, \ldots, M^{(\nu)}_d )$ on  $\mathcal{H}^{(\nu)}$ is bounded and homogeneous (cf. \cite{MisraBagchi}, \cite{Zhang}). One  can also verify that $\boldsymbol M^{(\nu)}$ is in $\mathcal A\mathbb{K}(\Omega)$. 
Replacing $(\nu)_{\s}$ by any arbitrary positive number $a_{\s}$ with some boundedness condition, we get a large class of operator tuples in $\mathcal A\mathbb{K}(\Omega)$ and we prove that they are the all.

To facilitate the study of $\mathbb{K}$-homogeneous operators,
we recall the following result from \cite{Arazy} describing all the $\mathbb{K}$-invariant kernels on $\Omega.$ 
\begin{proposition}[Proposiition 3.4, \cite{Arazy}]\label{proparazy}
	For any $\mathbb{K}$-invariant semi-inner product $\langle \cdot, \cdot \rangle $ on the the space of polynomials $\mathcal{P}$, we have the following.
	\begin{itemize}
		\item[\rm (i)]  $\mathcal{P}_{\underline{s}}$ is orthogonal to $\mathcal{P}_{\underline{s^\prime}}$  whenever $\underline s\neq \underline {s^\prime}$.
		\item[\rm (ii)] There exists a constant $b_{\underline{s}}\geq 0$  associated to each signature $\underline{s}$ such that
		$$\langle p,q\rangle= b_{\underline{s}}\langle p,q\rangle_{\mathcal{F}}, \qquad \mbox{for all}~ p,q\in \mathcal{P}_{\s}.$$ 
		\item[\rm (iii)]
		 $b_{\underline{s}} > 0$ for all $\underline{s}$ if and only if  $\inp{\cdot}{\cdot}$ is an inner product.
		\item[\rm (iv)]
		If the evaluation at each point of $\Omega$ is continuous on $(\mathcal{P}, \langle \cdot, \cdot \rangle)$ then the completion $\mathcal{H}$ of $(\mathcal{P}, \langle \cdot, \cdot \rangle)$ is a reproducing kernel Hilbert space. Moreover the kernel $K(\boldsymbol z,\boldsymbol w)$ is of the form   
		$$K(\boldsymbol z,\boldsymbol w)=\sum_{\underline{s}}b_{\underline{s}}^{-1}K_{\s} (\boldsymbol z, \boldsymbol w),$$
		where convergence is both pointwise and uniformly on compact subsets of $\Omega\times \Omega$ and in norm.	
	\end{itemize}
\end{proposition}

The following result is a generalization of \cite[Lemma 2.10]{CY} which is necessary  for the proof of Theorem \ref{thmKhomo} giving a model for commuting $d$-tuple of operators in the class $\mathcal A\mathbb K(\Omega)$.
\begin{lemma}
Let $\boldsymbol{T}=(T_1,T_2,\ldots,T_d)$ be a $\mathbb{K}$-homogeneous $d$-tuple of commuting operators on $\mathcal{H}$. Suppose that $\ker D_{\boldsymbol{T}^*}$ is one-dimensional and is spanned by a vector $e\in \mathcal{H}$ and that $e$ is cyclic for $\boldsymbol{T}$. Then there exists a sequence of non-negative real numbers $a_{\underline{s}}$, $\underline{s}  \in \overrightarrow  {\mathbb N}^r_0 $, such that for any polynomial $p\in \mathcal P$, 
\begin{equation}\label{eqnas}
{||p(\boldsymbol{T})e||}^2=\sum_{k=0}^{\text{deg}~ p}\sum_{|\underline{s}|=k}a_{\underline{s}}\|p_{\underline{s}}\|^2_{\mathcal F}
\end{equation}
where deg $p$ is the degree of $p$ and $$p=\sum_{k=0}^{\text{deg}~p}\sum_{|\underline{s}|=k}p_{\underline{s}}$$	is the Peter-Weyl decomposition.
\end{lemma}
\begin{proof}
Since $\boldsymbol{T}$ is $\mathbb{K}$-homogeneous, for each $k\in \mathbb{K}$ there exists a unitary operator $\Gamma(k)$ on $\mathcal H$ such that
\[ T_j\Gamma(k)=\Gamma(k)(k\cdot \boldsymbol{T})_j,~~j=1,\ldots,d.\]
Hence $T^*_j{\Gamma(k)}={\Gamma(k)}(k\cdot \boldsymbol{T})_j^*$, $j=1,\ldots,d.$ Since $(k\cdot \boldsymbol{T})_j$ is a linear combination of $T_1,\ldots,T_d$ and $e \in \ker D_{\boldsymbol{T}^*}$, it follows that 
${\Gamma(k})e$ belongs to  $\ker D_{\boldsymbol{T}^*}$ for all $k\in \mathbb{K}$.
Furthermore, since $\ker D_{\boldsymbol{T}^*}$ is one dimensional and spanned by $e$, we see that $\Gamma(k)e=\eta(k)e$ for some $\eta(k)$ such that $|\eta(k)|=1$. We now define a semidefinite sesquilinear product on $\mathcal{P}_{\underline{s}}$ for all $\s \in \overrightarrow  {\mathbb N}^r_0$ by the formula
	$$\langle p, q\rangle_{\mathcal{P}_{\underline{s}}}=\langle p(\boldsymbol{T})e,q(\boldsymbol{T})e\rangle_{\mathcal{H}}$$
	Now for any $k\in \mathbb{K}$ we have 
	\begin{align*} 
	 \langle p(k\cdot \boldsymbol z), q(k\cdot \boldsymbol z)\rangle_{\mathcal{P}_{\underline{s}}}
	&=\langle p(k\cdot \boldsymbol{T})e, q(k\cdot \boldsymbol{T})e\rangle_{\mathcal{H}}\\
	&=\langle \Gamma(k)^* p(\boldsymbol{T})\Gamma(k)e,\Gamma(k)^* q(\boldsymbol{T})\Gamma(k)e\rangle_{\mathcal{H}}\\
	&=\langle p(\boldsymbol{T})\Gamma(k)e, q(\boldsymbol{T})\Gamma(k)e\rangle_{\mathcal{H}}\\
	&=\langle  p(\boldsymbol{T})\eta(k)e, q(\boldsymbol{T})\eta(k)e\rangle_{\mathcal{H}}\\
	&=|\eta(k)|^2\langle p(\boldsymbol{T})e, q(\boldsymbol{T})e\rangle_{\mathcal{H}}\\
	&=\langle p(\boldsymbol{T})e, q(\boldsymbol{T})e\rangle_{\mathcal{H}}\\
	&=\langle p, q\rangle_{\mathcal{P}_{\underline{s}}}.
	\end{align*}
	So $\langle .,. \rangle_{\mathcal{P}_{\underline{s}}}$ is an $\mathbb{K}$-invariant semi-inner product on each ${\mathcal{P}_{\underline{s}}}$. Since $\{\mathcal{P}_{\underline{s}}\}_{\s \in \overrightarrow  {\mathbb N}^r_0}$ are mutually orthogonal spaces, it follows that  $\sum_{\s} \langle \cdot, \cdot \rangle_{\mathcal{P}_{\underline{s}}}$ defines a $\mathbb{K}$-invariant semi-inner product $\langle \cdot, \cdot \rangle$ on $\oplus_{\s}\mathcal{P}_{\s}$. Thus by Proposition \ref{proparazy}, there exists a non-negative constant $a_{\s}$ such that 
	$$\langle p,q \rangle_{\mathcal{P}_{\underline{s}}}=a_{\underline{s}} \langle p,q \rangle_{\mathcal{F}}.$$ This completes the proof of lemma.
	\end{proof}

   For all the classical domains, it can be easily verified that $\Omega =\{\boldsymbol w \in \mathbb C^d : \overline {\boldsymbol{w} }\in \Omega\}.$  Consequently, in the following theorem, the Hilbert space of holomorphic functions that we constructn live on $\Omega$ rather than $\{\boldsymbol w \in \mathbb C^d : \overline {\boldsymbol{w} }\in \Omega\}.$ The next result provides an analytic model for any $d$-tuple of operators $\boldsymbol T$ in $\mathcal {A}\mathbb K(\Omega)$.
\begin{theorem}\label{thmKhomo}
	If $\boldsymbol{T}$ is a $d$-tuple of operators in $\mathcal A \mathbb {K}(\Omega)$, then $\boldsymbol{T}$ is unitarily equivalent to a $d$-tuple $\boldsymbol{M}=(M_1, \ldots, M_d )$ of multiplication by the coordinate functions $z_1,\ldots ,z_d$ on a reproducing kernel Hilbert space $H_K$ of holomorphic functions defined on $\Omega$  with  $K(\boldsymbol z,\boldsymbol w)=\sum a_{\s}^{-1}K_{\s}(\boldsymbol z,\boldsymbol w)$ for all $\boldsymbol z, \boldsymbol w \in \Omega,$ for some choice of positive real numbers $a_{\underline{s}}$ with $a_{\underline{0}}=1$. 
\end{theorem}

\begin{proof}
	Since $\Omega\subseteq \sigma_p(\boldsymbol{T}^*),$ for each $\boldsymbol w \in \Omega$ there exists a non-zero vector $x \in \mathcal H,$ such that $T^*_j x=\bar{w}_j x$ for all $j=1,2, \ldots, d.$ Thus for any polynomial $p \in \mathcal P,$ we have $p(\boldsymbol{T}^*) x= p(\overline {\boldsymbol w})x$. If $e \in ker D_{\boldsymbol{T}^*}$ is a cyclic vector for $T$, then  
	\beqn p(\boldsymbol w)\inp{e}{x}_{\mathcal H} = \inp{e}{\overline{p(\boldsymbol w)}x}_{\mathcal H} = \inp{e}{\bar{p}(\boldsymbol{T}^*)x}_{\mathcal H} =\inp{p(\boldsymbol{T})e}{x}_{\mathcal H}.\eeqn
	Since $e$ is a cyclic vector for $T$ of norm $1$ and $x \neq 0$, we have that $\inp{e}{x}_{\mathcal H}\neq 0.$ Thus
	\beqn |p(w)|\leq \frac{\|p(\boldsymbol{T})e\|_{\mathcal H}\|x\|_{\mathcal H}}{|\inp{e}{x}_{\mathcal H}|}.\eeqn
	It follows that evaluation at $\boldsymbol w \in \Omega$ is bounded and therefore the semi-inner product defined by the rule $\inp{p}{q}_{\mathcal{P}_{\underline{s}}} =\inp{p(\boldsymbol{T})e}{q(\boldsymbol{T})e}_{\mathcal H}$ is an inner product on each $\mathcal P_{\underline{s}}.$  This gives rise to an inner product $\inp{\cdot}{\cdot}$ on the space of polynomials $\mathcal P$. The sequence $a_{\s}$ of Lemma 2.3, using Proposition 2.2(c), is now evidently positive. Moreover, since $\|e\|=1$, it follows from \eqref{eqnas} that $a_{\underline{0}}=1$. Thus the completion of  $(\mathcal P, \inp{\cdot}{\cdot})$, say $H_{K}$, is a reproducing kernel Hilbert space, where 
	\beqn K(\boldsymbol z,\boldsymbol w)=\sum a_{\s}^{-1}K_{\s}(\boldsymbol z,\boldsymbol w), \qquad \boldsymbol z,\boldsymbol w \in \Omega.\eeqn
	Clearly, the map $p \rightarrow p(\boldsymbol{T})e$  extends to a unitary from  $H_{K}$ to $\mathcal H,$ which intertwines $\boldsymbol{T}$ with the multiplication $d$-tuple $\boldsymbol{M}=(M_1, \ldots, M_d )$ on $H_K$.
\end{proof}
\begin{proposition}
If $\boldsymbol{T}$ is a $d$-tuple of operators in $\mathcal A \mathbb {K}(\Omega)$, then there exists a unitary representation $\Gamma:\mathbb{K} \to \mathcal U(\hl)$ such that  
$$\boldsymbol T\Gamma(k)=\Gamma(k)(k\cdot\boldsymbol{T}).$$
\end{proposition}
\begin{proof}
By Theorem \ref{thmKhomo},  $\boldsymbol{T}$ is unitarily equivalent to the $d$-tuple $\boldsymbol M=(M_1,\ldots,M_d)$ of multiplication operators on a reproducing kernel Hilbert space $H_K$ of holomorphic functions defined on $\Omega$  with  a kernel $K(z,w)$ which is $\mathbb{K}$-invariant.
 Clearly, the map $\Gamma$ on $\hl_K$ given by $\Gamma (k)(f)=f\circ k^{-1}(\cdot)$ is a unitary representation of $\mathbb{K}$ satisfying the intertwining condition.
\end{proof}

\begin{remark}
Since $\mathbb{K}$ is a subgroup of the group $\mathcal U(d)$  of unitary linear transformations on $\mathbb C^d$, every spherical $d$-tuple $\boldsymbol T=(T_1,\ldots, T_d)$ is $\mathbb{K}$-homogeneous.  Conversely, a $\mathbb{K}$-homogeneous $d$-tuple of theorem \ref{thmKhomo} is spherical if and only if $a_{\s}=a_{\underline {s}'}$ for all signatures $\s, \underline {s}'$ with $|\s|=|\underline {s}'|$.
\end{remark}

\begin{remark}
We also point out that, by the spectral mapping theorem, the Taylor joint
spectrum $\sigma(\boldsymbol T)$ of a $\mathbb K$-homogeneous operator $\boldsymbol T$ is $\mathbb K$-invariant, that is, if 
$\boldsymbol w$ belongs to $\sigma(\boldsymbol T)$, then $k.\boldsymbol w$ also belongs to $\sigma(\boldsymbol T)$ for all $k\in \mathbb K.$
\end{remark}

\section{Boundedness of the multiplication tuple}
Throughout the rest of the paper, let $K^{(a)}:\Omega \times \Omega \to \mathbb C$ denote the kernel function given by the  formula $K^{(a)}(\boldsymbol z,\boldsymbol w)=\sum_{\s} a_{\s} K_{\s}(\boldsymbol z,\boldsymbol w)$, $\boldsymbol z,\boldsymbol w\in \Omega$, for some choice of positive real numbers $a_{\s}$. The positivity  of the sequence $a_{\s}$ ensures that $K^{(a)}$ is a positive definite kernel. Thus it determines a unique Hilbert space $\mathcal H^{(a)}\subseteq \text{Hol}(\Omega)$ with the reproducing property: $\langle f, K^{(a)}(\cdot, \boldsymbol w) \rangle = f(\boldsymbol w)$, $f\in \mathcal H^{(a)}$, $\boldsymbol w\in \Omega$. It follows from Proposition \ref{proparazy} that the polynomial ring $\mathcal P$ is dense in $\mathcal H^{(a)}$ and $\mathcal P_{\s}$ is orthogonal to $\mathcal P_{\s^{\prime}}$
whenever $\underline s\neq \underline {s}^\prime$, that is,
$\mathcal H^{(a)}=\oplus_{\s\in \overrightarrow  {\mathbb N}^r_0} \mathcal P_{\s}.$
In this section, we discuss the boundedness of the $d$-tuple $\boldsymbol{M}^{(a)}:=(M_1^{(a)},\ldots, M_d^{(a)})$ of multiplication by the coordinate functions $z_1,\ldots ,z_d$ on the Hilbert space $\mathcal H^{(a)}$.
We begin with the following basic lemma, which is surely known to the 
experts, but we provide a proof for the sake of completeness.

\begin{lemma}\label{lemdiag}
The operators $\sum_{i=1}^d{M_i^{(a)}}^*M_{i}^{(a)}$ and $\sum_{i=1}^dM_i^{(a)}{ M_{i}^{(a)}}^*$, acting on the Hilbert space $\mathcal H^{(a)}$,
are block diagonal with respect to the decomposition $\oplus_{\s\in \overrightarrow  {\mathbb N}^r_0} \mathcal P_{\s}$, where each block is a non-negative scalar multiple of the identity operator. 
\end{lemma} 
\begin{proof} It is enough to give the proof for the operator $\sum_{i=1}^d{M_i^{(a)}}^*M_{i}^{(a)}$ since the proof for the operator $\sum_{i=1}^dM_i^{(a)}{ M_{i}^{(a)}}^*$ follows exactly in the same way. First, note that  $\Gamma(k)^*M_i^{(a)}\Gamma(k)=M_{z_i\circ k^{-1}}^{(a)}$ for $k\in  \mathbb{K}$.  Let $e_1, \ldots, e_d$ be the standard basis vectors in $\mathbb C^d$.  Note that   
$M_{z_i\circ k^{-1}}^{(a)}=\sum_{j=1}^d \langle k^{-1}e_j,e_i\rangle M_j^{(a)}$. In consequence, we have 
\begin{align*}
\Gamma(k)^*\left(\sum_{i=1}^d{M_i^{(a)}}^*M_{i}^{(a)}\right) \Gamma(k) 
& = \sum_{i=1}^d \Gamma(k)^*{M_i^{(a)}}^*\Gamma(k)\Gamma(k)^*M_i^{(a)}\Gamma(k)\\
& = \sum_{i=1}^d {M_{z_i\circ k^{-1}}^{(a)}}^*M_{z_i\circ k^{-1}}^{(a)}\\
& = \sum_{i=1}^d \sum_{p,q=1}^d \langle e_i, k^{-1}e_p\rangle \langle k^{-1}e_q, e_i \rangle {M_p^{(a)}}^*M_q^{(a)}\\
& = \sum_{p,q=1}^d    \langle k^{-1}e_q,  k^{-1}e_p\rangle {M_p^{(a)}}^*M_q^{(a)}\\
& = \sum_{i=1}^d{M_i^{(a)}}^*M_{i}^{(a)}.
\end{align*}
Here  the last equality follows from the fact that the subgroup $\mathbb{K}$
is contained in the group $\mathcal U(d)$ of unitary linear transformations on $\mathbb C^d$. Since $\mathcal H^{(a)}=\oplus_{\s\in \overrightarrow  {\mathbb N}^r_0} \mathcal P_{\s}$ and these subspaces $\{\mathcal P_{\s}\}$ are all the $\mathbb{K}$-irreducible subspaces of $\mathcal H^{(a)}$ and they are mutually $\mathbb K$-inequivalent, the conclusion follows from the Schur's lemma.
\end{proof}
The following lemma generalizes a known result \cite[Proposition 4.4]{Zhang} for the  multiplication $d$-tuple on the weighted Bergman spaces.
\begin{lemma}\label{lemarzyzhng}
For $f\in \mathcal P_{\s} $, we have $\sum_{i=1}^dM_i^{(a)}{ M_{i}^{(a)}}^*f=\tau(\s)f$, where
	\[\tau(\s)= \begin{cases}  \sum_{j=1}^{r}\frac{a_{\s-\epsilon_{j}}}{a_{\s}}\frac{(\frac{d}{r})_{\s}}{(\frac{d}{r})_{\s-\epsilon_{j}}}\frac{\frac{a}{2}(r-j)+s_{j}}{b+\frac{a}{2}(r-j)+s_{j}}\prod_{k\neq j}\frac{s_{j}-s_{k}+\frac{a}{2}(k-j-1)}{s_{j}-s_{k}+\frac{a}{2}(k-j)}& \mbox{if~~}\s\neq0, \\
0	& \mbox{if~~} \underline{s} = 0.\end{cases}\] 
\end{lemma}
The proof of the preceding lemma is very similar to the proof of \mbox{\cite[Proposition 4.4]{Zhang}} and therefore it is omitted.


For a signature $\s$,  in the remaining portion of this paper, we set 
\[c_{\s}(j)=\prod_{k\neq j}\frac{s_{j}-s_{k}+\frac{a}{2}(k-j+1)}{s_{j}-s_{k}+\frac{a}{2}(k-j)},~j=1,\ldots,r, \]
and \[ c^{\prime}_{\s}(j)=\prod_{k\neq j}\frac{s_{j}-s_{k}+\frac{a}{2}(k-j-1)}{s_{j}-s_{k}+\frac{a}{2}(k-j)},~j=1,\ldots,r.\]
If $\s+\epsilon_j$ is also a signature, then it is easy to see that $c_{\s}(j)>0$. Otherwise, 
$c_{\s}(j)=0$. Similarly, if $\s-\epsilon_j$ is  a signature, then $c_{\s}^{\prime}(j)>0$. Otherwise, 
$c_{\s}^{\prime}(j)=0$.

\begin{lemma}\label{lemcsjp}
For any fixed but arbitrary signature $\s $,  we have \[\sum_{j=1}^{r}c^{\prime}_{\s}(j)=\sum_{j=1}^{r}c_{\s}(j)=r.\]	
\end{lemma}
\begin{proof} Evidently, we have 
\begin{align*}
\sum_{j=1}^{r}c^{\prime}_{\s}(j)&=\sum_{j=1}^{r}\prod_{k\neq j}\frac{s_{j}-s_{k}+\frac{a}{2}(k-j-1)}{s_{j}-s_{k}+\frac{a}{2}(k-j)}\\
&=\sum_{j=1}^{r}\prod_{k\neq j}\bigg(1-\frac{\frac{a}{2}}{s_{j}-s_{k}+\frac{a}{2}(k-j)}\bigg)\\
&=\sum_{j=1}^{r}\prod_{k\neq j}\bigg(1-\frac{\frac{a}{2}}{(s_{j}-\frac{a}{2}j)-(s_{k}-\frac{a}{2}k)}\bigg).
\end{align*}
Setting $s_j^{\prime}=\frac{s_j-\frac{a}{2}j}{\frac{a}{2}}$, we see that $s_1^{\prime}>s_2^{\prime}>\cdots>s_r^{\prime}$, and 
\begin{align*}
\sum_{j=1}^{r}c^{\prime}_{\s}(j)=\sum_{j=1}^{r}\prod_{k\neq j}\Big(1-\frac{1}{s_j^\prime-s_k^\prime}\Big)
&=r+\sum_{j=1}^r\sum_{\substack{A\subseteq\{1,\ldots,j-1,j+1,\ldots,r\}\\ A\neq \phi}}(-1)^{|A|}\prod_{k\in A} \frac{1}{{s_j^\prime-s_k^\prime}}\\
&=r+\sum_{\substack{A\subseteq\{1,\ldots,r\}\\ |A|\geq 2}}(-1)^{|A|-1}\sum_{j\in A} \prod_{\substack{k\in A\\ k\neq j}}\frac{1}{{s_j^\prime-s_k^\prime}}.
\end{align*}
Now, by \cite[Corollary 2.3]{MisraBagchi}, it follows that 
$\sum_{j\in A} \prod_{\substack{k\in A\\ k\neq j}}\frac{1}{{s_j^\prime-s_k^\prime}}=0$ for all $A\subseteq\{1,\ldots,r\}$ with $|A|\geq 2$. Therefore, $\sum_{j=1}^{r}c^{\prime}_{\s}(j)=r$. The proof of the other part follows exactly in the same way.
\end{proof}
\begin{theorem}
The $d$-tuple $\boldsymbol M^{(a)} = (M_1^{(a)},\ldots, M_d^{(a)})$ of multiplication operators on the Hilbert space  $\mathcal H^{(a)}$ is bounded if and only if
\[A:=\sup\left\{\frac{a_{\s-\epsilon_{j}}}{a_{\s}}\frac{(\frac{d}{r})_{\s}}{(\frac{d}{r})_{\s-\epsilon_{j}}} 
:~ \s, \s-\epsilon_j \in \overrightarrow  {\mathbb N}^r_0,~ j=1,\ldots,r	
\right\}\] is finite.
\end{theorem}
\begin{proof}
Clearly, the multiplication $d$-tuple $\boldsymbol M^{(a)}$ on the Hilbert
space $\mathcal H^{(a)}$ is bounded if and only if $\sum_{i=1}^dM_i^{(a)} {M_{i}^{(a)}}^*$ is bounded. Therefore, using Lemma \ref{lemarzyzhng}, it is enough to show that $\tau(\s)$ is bounded for all $\s \in \overrightarrow  {\mathbb N}^r_0$ if and only if $A$ is finite. First assume  that $A$ is finite. Then  
\begin{align*}
\tau(\s)&=\sum_{j=1}^{r}\frac{a_{\s-\epsilon_{j}}}{a_{\s}}\frac{(\frac{d}{r})_{\s}}{(\frac{d}{r})_{\s-\epsilon_{j}}}\frac{\frac{a}{2}(r-j)+s_{j}}{b+\frac{a}{2}(r-j)+s_{j}}c^{\prime}_{\s}(j)\\
&\leq A\sum_{j=1}^{r} \frac{\frac{a}{2}(r-j)+s_{j}}{b+\frac{a}{2}(r-j)+s_{j}}c^{\prime}_{\s}(j)\\
&\leq A\sum_{j=1}^{r}c^{\prime}_{\s}(j) \\
&=Ar .
\end{align*}
for any signature $\s$.  Here, the last equality follows from Lemma \ref{lemcsjp}.
To prove the other direction, assume that 
$\tau(\s)$ is bounded, that is, $\tau(\s)\leq B$ for some positive real number $B$ and for all $\s \in \overrightarrow  {\mathbb N}^r_0$. Thus  
\[\frac{a_{\s-\epsilon_{j}}}{a_{\s}}\frac{(\frac{d}{r})_{\s}}{(\frac{d}{r})_{\s-\epsilon_{j}}}\frac{\frac{a}{2}(r-j)+s_{j}}{b+\frac{a}{2}(r-j)+s_{j}}c^{\prime}_{\s}(j)\leq \tau(\s)\leq B.\]
Now, note that if $\s-\epsilon_j\in \overrightarrow  {\mathbb N}^r_0,$ then
\begin{align}
\frac{1}{c_{\s}^{\prime}(j)}&= \prod_{k\neq j}\frac{s_{j}-s_{k}+\frac{a}{2}(k-j)}{s_{j}-s_{k}+\frac{a}{2}(k-j-1)}\nonumber\\
&=\prod_{k<j}\frac{s_{j}-s_{k}+\frac{a}{2}(k-j)}{s_{j}-s_{k}+\frac{a}{2}(k-j-1)}\prod_{k> j}\frac{s_{j}-s_{k}+\frac{a}{2}(k-j)}{s_{j}-s_{k}+\frac{a}{2}(k-j-1)}\nonumber\\
&\leq \prod_{k> j}\frac{s_{j}-s_{k}+\frac{a}{2}(k-j)}{s_{j}-s_{k}+\frac{a}{2}(k-j-1)}\nonumber\\
& \leq \prod_{k> j}\frac{s_{j}-s_{k}+\frac{a}{2}(k-j)}{s_{j}-s_{k}}\nonumber\\
& \leq \prod_{k> j}(1+\frac{\frac{a}{2}(k-j)}{s_{j}-s_{k}})\nonumber\\
& \leq (1+\frac{a}{2}(r-1))^r\label{eqncsjp}.
\end{align}
Here the third inequality holds since $\frac{s_{j}-s_{k}+\frac{a}{2}(k-j)}{s_{j}-s_{k}+\frac{a}{2}(k-j-1)}\leq 1$  for $k<j$.
Finally, see that
\begin{equation*}
\frac{a_{\s-\epsilon_{j}}}{a_{\s}}\frac{(\frac{d}{r})_{\s}}{(\frac{d}{r})_{\s-\epsilon_{j}}}  \leq \frac{B}{c_{\s}^{\prime}(j)}\frac{b+\frac{a}{2}(r-j)+s_{j}}{\frac{a}{2}(r-j)+s_{j}}
\leq B(1+\frac{a}{2}(r-1))^r(1+b).
\end{equation*}
This completes the proof.
\end{proof}

\begin{corollary}
The multiplication $d$-tuple $\boldsymbol{ M}^{(\nu)}$  on the Hilbert space 
$\mathcal H^{(\nu)}$ is bounded if $\nu>\frac{a}{2}(r-1).$
\end{corollary}
\begin{proof}
If $\nu>\frac{a}{2}(r-1),$ then 
\[\frac{(\nu)_{\s-\epsilon_{j}}}{(\nu)_{\s}}\frac{(\frac{d}{r})_{\s}}{(\frac{d}{r})_{\s-\epsilon_{j}}}=\frac{\frac{d}{r}-\frac{a}{2}(j-1)+s_{j}-1}{\nu-\frac{a}{2}(j-1)+s_{j}-1}\leq \max\Big \{1,\frac{1+b}{\nu-\frac{a}{2}(r-1)} \Big \}.\]
Therefore, from Theorem 3.4, it follows that $\boldsymbol{M}^{(\nu)}$  is bounded. 
\end{proof}

Having determined (a) the condition for boundedness of the operator $\boldsymbol M^{(a)}$, (b) noting that each $\boldsymbol w$ in $\Omega$ is a joint eigenvector for the multiplication $d$-tuple $\boldsymbol M^{(a)*}$ and finally since the constant vector $1$ is cyclic for $\boldsymbol M^{(a)}$, it is natural to investigate the question of which of these are in the Cowen-Douglas class $\text{B}_1(\Omega)$, see \cite{CD}, \cite{CDopen} for the definition of this very important class of operators. As shown in \cite[pp. 285]{equivofquotient}, the cyclicity implies that the dimension of the joint eigenspace at each $\boldsymbol w$ in $\Omega$ is $1$. Thus to determine the membership in the Cowen-Douglas class in a neighbourhood of origin contained in $\Omega$, we only need to find when $\mbox{\rm ran}D_{{\boldsymbol M^{(a)}}^*}$ is closed. The following theorem provides the precise condition for this.  
\begin{theorem}\label{ranclosed}
For a multiplication $d$-tuple $\boldsymbol M^{(a)}=(M_1^{(a)},\ldots, M_d^{(a)})$ on $\mathcal H^{(a)}$, $\mbox{\rm ran}D_{{\boldsymbol{M}^{(a)}}^*}$ is closed if
and only if \[B:=\inf\left\{\sum_{j=1}^{r}\frac{a_{\s-\epsilon_j}}{a_{\s}}\frac{(\frac{d}{r})_{\s}}{(\frac{d}{r})_{\s-\epsilon_j}} 
:~ \s, \s-\epsilon_j\in \overrightarrow  {\mathbb N}^r_0,~ j=1,\ldots,r	
\right\}\] is non-zero positive.
\end{theorem}
\begin{proof}	
It is elementary to see that $\mbox{\rm ran}D_{{\boldsymbol{M}^{(a)}}^*}$ is closed if and only if $\sum_{i=1}^dM_i^{(a)} {M_{i}^{(a)}}^*$ is bounded below on $\ker D_{{\boldsymbol{M}^{(a)}}^*}^{\bot}$. Also, for the $d$-tuple $\boldsymbol{M}^{(a)}$	on $\mathcal H^{(a)}$, we have $\ker D_{{\boldsymbol{M}^{(a)}}^*}=\mathcal P_{0}$, the space of constant functions. Therefore, in view of Lemma \ref{lemarzyzhng}, it suffices to show that $B$ is non-zero positive if and only if 
$\inf\{\tau(\s): \s\neq 0\}$ is non-zero positive.
Suppose that $B$ is 
a non-zero positive number.  
Now, for any signature $\s\neq0$, we have
\begin{align*}
\tau(\s)&=\sum_{j=1}^{r}\frac{a_{\s-\epsilon_{j}}}{a_{\s}}\frac{(\frac{d}{r})_{\s}}{(\frac{d}{r})_{\s-\epsilon_{j}}}\frac{\frac{a}{2}(r-j)+s_{j}}{b+\frac{a}{2}(r-j)+s_{j}}c_{\s}^{\prime}(j)\\
&\geq\frac{1}{b+1}\sum_{j=1}^{r}\frac{a_{\s-\epsilon_{j}}}{a_{\s}}\frac{(\frac{d}{r})_{\s}}{(\frac{d}{r})_{\s-\epsilon_{j}}}c^{\prime}_{\s}(j)\\
&\geq\frac{1}{b+1}\sum_{j=1}^{r}\frac{a_{\s-\epsilon_{j}}}{a_{\s}}\frac{(\frac{d}{r})_{\s}}{(\frac{d}{r})_{\s-\epsilon_{j}}}\frac{1}{(1+\frac{a}{2}(r-1))^r}\\
&\geq \frac{B}{(b+1)(1+\frac{a}{2}(r-1))^r}.
\end{align*}
Here the third inequality follows from $\eqref{eqncsjp}$.
Conversely, assume that $\inf\{\tau(\s): \s\neq 0\}$ is a non-zero positive number, say $C$. Thus for each signature $\s\neq0$
\begin{equation}
\sum_{j=1}^{r}\frac{a_{\s-\epsilon_{j}}}{a_{\s}}\frac{(\frac{d}{r})_{\s}}{(\frac{d}{r})_{\s-\epsilon_{j}}}\frac{\frac{a}{2}(r-j)+s_{j}}{b+\frac{a}{2}(r-j)+s_{j}}c^{\prime}_{s}(j)\geq C.
\end{equation}
Hence, noting that $c^{\prime}_{s}(j)\leq r$ by Lemma  \ref{lemcsjp} and  $\frac{\frac{a}{2}(r-j)+s_{j}}{b+\frac{a}{2}(r-j)}\leq 1$, it follows that 
\[\sum_{j=1}^r\frac{a_{\s-\epsilon_{j}}}{a_{\s}}\frac{(\frac{d}{r})_{\s}}{(\frac{d}{r})_{\s-\epsilon_{j}}}\geq \frac{C}{r}.\]
\end{proof}

\begin{corollary}
	The  range  of $D_{\boldsymbol M^{(\nu)^*}}$ is closed if $\nu>\frac{a}{2}(r-1).$
\end{corollary}
\begin{proof}
	Suppose $\nu=\frac{a}{2}(r-1)+ \epsilon$ for some  $\epsilon > 0$. Then
	\[\sum_{j=1}^{r}\frac{(\nu)_{\s-\epsilon_j}}{(\nu)_{\s}}\frac{(\frac{d}{r})_{\s}}{(\frac{d}{r})_{\s-\epsilon_{j}}}=\sum_{j=1}^{r}\frac{b+\frac{a}{2}(r-j)+s_{j}}{\frac{a}{2}(r-j)+s_{j}+\epsilon-1}\] which is always bounded below by 1 if $\epsilon\leq b+1$. On the other hand, for $\epsilon\geq b+1$, it is bounded below by $\frac{1}{\epsilon}$. Hence, by Theorem \ref{ranclosed}, 
	$\mbox{\rm ran}D_{\boldsymbol M^{(\nu)^*}}$ is closed. 
\end{proof}

Now, we wish to show that the adjoint ${\boldsymbol M^{(\nu)}}^*$ of the $d$-tuple of multiplication operators on $\hl^{(\nu)}$ is in the Cowen-Douglas class $\text{B}_1(\Omega)$ for $\nu >\tfrac{a}{2}(r-1)$.  

Recall that the left essential spectrum $\pi_e^{\ell,0}(\boldsymbol T)$ of a commuting $d$-tuple of operators $\boldsymbol T$ is defined to be the complement of the set of all $\boldsymbol w\in \mathbb C^d$ with the property:
\begin{enumerate}
\item $\dim \ker D_{(\boldsymbol T-\boldsymbol w I)}$ is finite, 
\item $\text{ran} D_{(\boldsymbol T-\boldsymbol w I)}$ is closed. 
\end{enumerate}

If $0\not \in \pi^{\ell,0}_e(\boldsymbol T)$, then  the $d$-tuple $\boldsymbol T$ is said to be left semi-Fredholm. 

The essential ingredient of the proof of the following theorem is based on the spectral mapping property of the left essential spectrum, which appears in \cite{EPut} and was pointed out by. J. Eschmeier to G. Misra during a conversation at University of Saarbrucken in February 2014. 

\begin{theorem}
The adjoint ${\boldsymbol M^{(\nu)}}^*$ of the multiplication $d$-tuple  on $\hl^{(\nu)}$ is in the Cowen-Douglas class $\text{B}_1(\Omega)$ whenever $\nu > \tfrac{a}{2} (r-1)$.
\end{theorem}
\begin{proof}
Since polynomials are dense in the Hilbert space $\mathcal H^{(\nu)}$, it follows that  $\dim \ker D_{{\boldsymbol M^{(\nu)}}^*}$ is $1$. By Corollary 3.7, we also have that $\text{ran}~ D_{{\boldsymbol M^{(\nu)}}^*}$ is closed. Therefore, $D_{{\boldsymbol M^{(\nu)}}^*}$ is left semi-Fredholm and hence there is a $\varepsilon > 0$ such that for $\boldsymbol w\in \Omega$ with $\|\boldsymbol  w\|_2 < \varepsilon$, the operators 
$D_{(\boldsymbol M - \boldsymbol w I)^*}$ are left Fedholm. Thus  ${\boldsymbol M^{(\nu)}}^*$ is in the Cowen-Douglas class $\text{B}_1(\Omega_\varepsilon)$, where $\Omega_\varepsilon=\{\boldsymbol w\in \Omega: \|\boldsymbol w\|_2 < \varepsilon\}$. Now, using the homogeneity of $\boldsymbol M^{(\nu)}$ and the spectral mapping property of the left essential spectrum, we show that ${\boldsymbol M^{(\nu)}}^*$ is actually in $\text{B}_1(\Omega)$. 

To complete the proof, first note that if $\boldsymbol w\in \Omega$ is any fixed but arbitrary point, then there exists a bi-holomorphic automorphism $\varphi$ of $\Omega$ with the property: $\varphi(0)=\boldsymbol w$. We have seen that $0\not \in \pi_e^{\ell,0}({\boldsymbol M^{(\nu)}}^*)$. An analytic spectral mapping property for the left essential spectrum is ensured by  \cite[Corollary 2.6.9]{EPut}.   It follows that 
$$\boldsymbol w=\varphi(0) \not \in \varphi (\pi_e^{\ell,0}({\boldsymbol M^{(\nu)}}^*))= \pi_e^{\ell,0}(\varphi({\boldsymbol M^{(\nu)}}^*)) = \pi_e^{\ell,0}({\boldsymbol M^{(\nu)}}^*).$$
Here the last equality follows from the homogeneity assumption. 
\end{proof}

\section{Computation of the operator $\sum M_i^*M_i$ on $\mathcal H^{(a)}$}
In this section, we wish to compute the operator ${\boldsymbol{M}^{(a)}}^*{\boldsymbol{M}^{(a)}}:=\sum_{i=1}^d {M_i^{(a)}}^*M_i^{(a)}$ on the Hilbert space $\hl^{(a)}$. First, we note that the bounded symmetric domain $\Omega$ sits inside a  linear space of dimension $d$ in its Harish-Chandra realization. 
The type I domains are realized as the open unit ball, with respect to the operator norm,  in the linear space of $n\times m$ matrices. The situation becomes somewhat different when we consider domains of type II. Pick one of these domains of dimension $\tfrac{n(n+1)}{2}$. It is convenient to put $\tfrac{n(n+1)}{2}$ variables in  the form of a symmetric matrix, where  the inner  product is given  by $\text{tr}(AB^*)$. Now, in  the  space of these symmetric matrices of size $n$, the matrices  $E_{ii}$, $i=1,\ldots,n$ together with 
$\tfrac{E_{ij} + E_{ji}}{\sqrt{2}}$, $1\leq  i\not = j \leq  n$ form an orthonormal basis.  Consequently, the coordinates of this domain is of the form $z_{11}, \sqrt{2} z_{12}, \ldots,  \sqrt{2} z_{1n}, z_{22}  \ldots \sqrt{2} z_{2n}, \ldots  , z_{n-1 n-1}, \sqrt{2} z_{n-1 n},  z_{nn}$, see \cite[pp. 130]{Hua}. One may pick coordinates similarly for the type III domains consisting of the $n\times n$ antisymmetric matrices of norm at  most $1$. Finally, the type IV domains, in its Harish-Chandra realization are described in \cite[pp. 76]{Mok}:
$$\left \{\boldsymbol z:=(z_1, \ldots , z_d): \sum_{i=1}^d|z_i|^2 < 2~\text{and}~\sum_{i=1}^d|z_i|^2  <  1+\left |\frac{1}{2} \sum_{i=1}^d  z_i^2\right |^2\right \}.$$

The following theorem appears in \cite{At} in a slightly different form. The difference  arises since we take the multiplication by the coordinate functions to be the ones described in the previous paragraph, while in the paper \cite{At}, these are the usual coordinates. Thus it makes no difference in the case of the type I domains, while for the other domains, the answer is different. 
 
\begin{theorem}\label{formulahardy} Let $\boldsymbol M^{(S)}=(M^{(S)}_1, \ldots, M^{(S)}_d)$ be the $d$-tuple of multiplication by the coordinate functions $z_1,\ldots ,z_d$ on the Hardy space $H^2(S),$ where $S$ is the Shilov boundary  of $\Omega$. Then
\begin{equation}
\sum_{i=1}^d {M_i^{(S)}}^*M_i^{(S)} =rI.
\end{equation}
\end{theorem}

By Lemma \ref{lemdiag},  note that 
${\boldsymbol{M}^{(a)}}^*{\boldsymbol{M}^{(a)}}$ is a block diagonal operator with respect to the decomposition $\oplus \mathcal P_{\s}$, where each block is a non-negative scalar multiple of the identity, that is, ${\boldsymbol{M}^{(a)}}^*{\boldsymbol{M}^{(a)}} p=\delta(\s)p,~p\in \mathcal P_{\s}$ for some positive real number $\delta(\s)$. Therefore, we need to compute the constant $\delta(\s)$ for all $\s \in \overrightarrow  {\mathbb N}^r_0$. Unfortunately, we are only able to find $\delta(\s)$ when $\s$ is a signature such that  $\s+\eps_j$ is a signature for at most two $j$, where $1\leq j \leq r$. In particular, we have the complete answer in case the rank $r=2$.


\begin{proposition}\label{adjointofM}
For any polynomial $p\in \mathcal P_{\s}$, 
\[{M_i^{(a)}}^*p=\sum_{j=1}^r \frac{a_{\s-\epsilon_j}}{a_{\s}} (\partial_i p)_{\s-\epsilon_j}.\]
\end{proposition}
\begin{proof}
	By theorem \cite[4.11.86]{upmeier}, we see that $z_i\mathcal P_{\s}$ is contained in $\oplus_{j=1}^r \mathcal{P}_{\s+\eps_j}$. Thus, for any polynomial $p$ in $\mathcal P_{\s}$, it follows that  $M_i^*p$ belongs to $\oplus_{j=1}^r \mathcal P_{\s-\eps_j}$. Now for $q\in \mathcal P_{\s-\eps_j}$, we have
	\begin{align*}
	\langle M_i^*p,q\rangle_{\mathcal H^{(a)}}  = \langle p, z_i q\rangle_{\mathcal H^{(a)}}
	&=\langle p, (z_iq)_{\s}\rangle_{\mathcal H^{(a)}}\\
	&=\frac{1}{a_{\s}}\langle p , (z_iq)_{\s}\rangle_{\mathcal H^{(a)}}\\
	&=\frac{1}{a_{\s}}\langle p , z_iq\rangle_{\mathcal F}\\
	&=\frac{1}{a_{\s}}\langle \partial_i p , q\rangle_{\mathcal F}\\
	&=\frac{1}{a_{\s}}\langle (\partial_i p)_{\s-\eps_j} , q\rangle_{\mathcal F}\\
	&=\frac{a_{\s-\eps_j}}{a_{\s}}\langle (\partial_i p)_{\s-\eps_j} , q\rangle_{\mathcal H^{(a)}}.
	\end{align*}
	Therefore the conclusion follows.
\end{proof}

\begin{theorem}\label{thmsum}
\rm (i)
Let $\s$ be a signature such that $\s+\eps_j$ is a signature if and only if $j=1$. Then ${\boldsymbol{M}^{(a)}}^*{\boldsymbol{M}^{(a)}} p=\delta(\s)p,~p\in \mathcal P_{\s}$, where
\[\delta(\s)= \frac{a_{\s}}{a_{\s+\eps_1}}r(\tfrac{d}{r}+s_1).\]\\
\rm (ii) Let $\s$ be a signature such that $\s+\eps_j$ is a signature if and only if $j=1, k$, where $2\leq k\leq r$. Then ${\boldsymbol{M}^{(a)}}^*{\boldsymbol{M}^{(a)}} p=\delta(\s)p,~p\in \mathcal P_{\s}$, where
\[ \delta(\s)= 
   \frac{a_{\s}}{a_{\s+\eps_1}}\frac{(k-1)(\tfrac{d}{r}+s_1)(s_1-s_k+\tfrac{ar}{2})}{(s_1-s_k+\tfrac{a}{2}(k-1))}+
    \frac{a_{\s}}{a_{\s+\eps_k}}\frac{(r-k+1)(\tfrac{d}{r}-\tfrac{a}{2}(k-1)+s_k)(s_1-s_k)}{(s_1-s_k+\tfrac{a}{2}(k-1))}.
\]
\end{theorem}
\begin{proof}
First note that, for $p\in \mathcal P_{\s}$,  we have
\begin{align*}
\sum_{i=1}^d {M_i^{(a)}}^*M_i^{(a)} p  = \sum_{i=1}^d {M_i^{(a)}}^*(z_ip)
& = \sum_{i=1}^d \Big({M_i^{(a)}}^*\big( \sum_{j=1}^r (z_ip)_{\s+\epsilon_j}\big) \Big)_{\s}\\
& = \sum_{i=1}^d\Big(\sum_{j=1}^r \frac{a_{\s}}{a_{\s+\epsilon_j}}
\partial_i \big((z_ip)_{\s+\epsilon_j}\big) \Big)_{\s}  \\
& = \sum_{j=1}^r  \frac{a_{\s}}{a_{\s+\epsilon_j}} \sum_{i=1}^d \Big(\partial_i \big((z_ip)_{\s+\epsilon_j}\big)\Big)_{\s},
\end{align*}
where the third equality follows from Theorem \ref{adjointofM}.
Let $Q_j$ be the linear map on the space of polynomials given by
\[Q_j(p)=\sum_{i=1}^d \Big(\partial_j \big((z_ip)_{\s+\epsilon_j}\big)\Big)_{\s},\;\;p\in \mathcal P_{\s}.\]
Then clearly, 
\begin{equation}\label{eqnsum1}
\delta(\s)p=\sum_{j=1}^r \frac{a_{\s}}{a_{\s+\epsilon_j}}Q_j(p).
\end{equation}
Note that, for $p\in \mathcal P_{\s}$, 
$Q_j$ satisfies the following:
\begin{align*}
\sum_{j=1}^r Q_j(p)=\sum_{i=1}^d\sum_{j=1}^r \Big(\partial_i \big((z_ip)_{\s+\epsilon_j}\big)\Big)_{\s}
&=\sum_{i=1}^d \Big(\partial_i \big(\sum_{j=1}^r(z_ip)_{\s+\epsilon_j}\big)\Big)_{\s}\\
&=\sum_{i=1}^d \Big(\partial_i \big(z_ip\big)\Big)_{\s}\\
&=dp+\sum_{i=1}^d\Big(z_i\partial_i p\Big)_{\s}.
\end{align*}
Therefore, by Euler's formula
\begin{equation}\label{eqnsum2}
\sum_{j=1}^r Q_j(p)=(d+|\s|)p.
\end{equation}
If $\s$ is a signature such that $\s+\epsilon_j$ is a signature if and only if $j=1$, then it follows easily from \eqref{eqnsum1} and 
\eqref{eqnsum2} that $\delta(\s)=\frac{a_{\s}}{a_{\s+\eps_1}}r(\tfrac{d}{r}+s_1)$, proving the first part of the theorem. 

To prove the second part, note that by Theorem \ref{formulahardy}, we have 
\begin{equation}
\sum_{j=1}^r \frac{(\frac{d}{r})_{\s}}{(\frac{d}{r})_{\s+\epsilon_j}} Q_j (p)=rp. 
\end{equation}
If $\s$ is a signature such that $\s+\eps_j$ is a signature if and only if $j=1, k$, where $2\leq k\leq r$, then in the summation in \eqref{eqnsum1} and  \ref{eqnsum2}, only two terms survive, namely, $Q_1(p)$ and $Q_k(p)$. By solving these two equations, it is easily verified that 
\[Q_1(p)=\frac{(k-1)(\tfrac{d}{r}+s_1)(s_1-s_k+\tfrac{ar}{2})}{(s_1-s_k+\tfrac{a}{2}(k-1))} p,\]
and
\[Q_k(p)=\frac{(r-k+1)(\tfrac{d}{r}-\tfrac{a}{2}(k-1)+s_k)(s_1-s_k)}{(s_1-s_k+\tfrac{a}{2}(k-1))} p.\]
This completes the proof.
\end{proof}

\begin{corollary}
Let $\nu> \frac{a}{2}(r-1)$ and $\boldsymbol M^{(\nu)}= (M^{(\nu)}_1, \ldots, M^{(\nu)}_d )$ be the $d$-tuple of  multiplication operators on  $\mathcal{H}^{(\nu)}$. Then the operator $M_i^{(\nu)}$ is essentially normal, that is, the commutator ${M_i^{(\nu)}}^*M_i^{(\nu)}-M_i^{(\nu)}{M_i^{(\nu)}}^*$ is compact for all $i=1,\ldots, d$ if and only if $r=1$. 
\end{corollary}
\begin{proof}
If $r=1$, then by a direct computation it is easily verified that
each $M_i^{(\nu)}$ is essentially normal. For the converse part, first set $\underline{\boldsymbol{l}}$ to be the signature $(l,0,\ldots,0)$, where  $l$ is a positive integer. Then, by Lemma \ref{lemarzyzhng} and Theorem \ref{thmsum}, we see that,  $\sum_{i=1}^d({M_i^{(\nu)}}^*M_i^{(\nu)}-M_i^{(\nu)}{M_i^{(\nu)}}^*) p= \eta(\underline{\boldsymbol{l}}) p$, $p\in \mathcal P_{\underline{\boldsymbol{l}}}$, where 
\begin{equation}\label{eqnessnormal}
\eta(\underline{\boldsymbol{l}})=\frac{(\frac{d}{r}+l)(l+\frac{ar}{2})}{(\nu+l)(l+\frac{ar}{2})}+\frac{l(r-1)(\frac{d}{r}-\frac{a}{2})}{(\nu-\frac{a}{2})(l+\frac{a}{2})}-\frac{l}{\nu+l-1}.
\end{equation}
Suppose that each $M_i^{(\nu)}$ is essentially normal. Then the operator $\sum_{i=1}^d({M_i^{(\nu)}}^*M_i^{(\nu)}-M_i^{(\nu)}{M_i^{(\nu)}}^*)$ is compact. Hence $\eta(\underline{\boldsymbol{l}})$ must converge to $0$ as $l\to \infty$. Thus, from \eqref{eqnessnormal}, we obtain that $\frac{(r-1)(\frac{d}{r}-\frac{a}{2})}{\nu-\frac{a}{2}}=0.$ Finally, since $\frac{d}{r}=(r-1)\frac{a}{2}+b+1$, we get that $r=1.$
\end{proof}

\begin{conjecture}
For $p \in \mathcal P_{\s}$, $\sum_{i=1}^d M_i^*M_i p=\delta(\s) p$ on the Hilbert space $\hl^{(a)}$, where
\begin{equation}
\delta(\s)=\sum_{j=1}^r \frac{a_{\s}}{a_{\s+\epsilon_j}}\frac{(\frac{d}{r})_{\s+\epsilon_j}}{(\frac{d}{r})_{\s}}
\prod_{k\neq j}\frac{s_j-s_k+\frac{a}{2}(k-j+1)}{s_j-s_k+\frac{a}{2}(k-j)}.
\end{equation}
\end{conjecture}

\section{Unitary equivalence and Similarity}
In this section, we study the question of unitary equivalence and similarity of two commuting $d$-tuple of  operators in the class $\mathcal A\mathbb{K}(\Omega)$. In the particular case when $\mathbb{K}$ is the unit circle $\mathbb T$, these results were obtained by Shields in \cite{Sh}. 
The higher-dimensional counterpart of similarity result is obtained in \cite[Lemma 2.2]{Ku}.

By Theorem \ref{thmKhomo}, any $d$-tuple of  operator $\boldsymbol T$ in $\mathcal A\mathbb{K}(\Omega)$ is unitarily equivalent to $\boldsymbol{M}^{(a)}$  consisting of multiplication by the coordinate functions $z_1, \ldots, z_d$ on the reproducing kernel Hilbert space $\hl^{(a)}$ with the reproducing kernel 
$K^{(a)}(\boldsymbol z,\boldsymbol w)=\sum_{\s} a_{\s} K_{\s}(\boldsymbol z,\boldsymbol w)$, where $a_{\s}>0$ with $a_{\underline{0}}=1$.  Thus we assume, without loss of generality, that $\boldsymbol T\sim_u \boldsymbol{M}^{(a)}$.

\begin{theorem}
Let $\boldsymbol T_1$ and $\boldsymbol T_2$ be two $\mathbb K$-homogeneous  operator tuples in $\mathcal A\mathbb{K}(\Omega)$. Suppose that $T_1\sim_u \boldsymbol{M}^{(a)}$  and $T_2\sim_u  \boldsymbol{M}^{(b)}$.
Then the following statements are equivalent.
\begin{itemize}
\item[\rm (i)]$\boldsymbol T_1$ and $\boldsymbol T_2$ are unitarily equivalent.
\item[\rm (ii)]$a_{\s}=b_{\s}$ for all $\s \in \overrightarrow  {\mathbb N}^r_0$.
\item[\rm (iii)]$K^{(a)}=K^{(b)}.$
\end{itemize}
\end{theorem}
\begin{proof} It 
is easy to see that  \rm(ii) and  \rm(iii) are equivalent. It is obvious that 
\rm(iii) implies \rm(i). Therefore it remains to verify that \rm (i) implies \rm(iii). Assume that 
the $d$-tuples $\boldsymbol T_1$ and $\boldsymbol T_2$ are unitarily equivalent. Then so are the operators $\boldsymbol{M}^{(a)}$ and $\boldsymbol{M}^{(b)}$. By \cite[Theorem 3.7]{Curtosalinas}, there exists a holomorphic 
function $g$ on $\Omega$ such that 
$$K^{(a)}(\boldsymbol z,\boldsymbol w)=g(\boldsymbol z)K^{(b)}(\boldsymbol z,\boldsymbol w) \overline {g(\boldsymbol w)},\qquad \boldsymbol z,\boldsymbol w\in \Omega.$$
In particular, $K^{(a)}(\boldsymbol z,0)=g(\boldsymbol z)K^{(b)}(\boldsymbol z, 0)\overline {g(0)}$, $\boldsymbol z\in \Omega$. Therefore, $a_{\underline{0}}=b_{\underline{0}}g(\boldsymbol z)\overline{g(0)}$, and consequently, $g(z)\overline{g(0)}=1$ since $a_{\underline{0}}=b_{\underline{0}}=1$. Hence $K^{(a)}=K^{(b)}.$ 
\end{proof}

Recall that two commuting $d$-tuples $\boldsymbol A=(A_1,\ldots,A_d)$ and $\boldsymbol B=(B_1,\ldots,B_d)$, defined on $\hl_1$ and $\hl_2$ respectively, are said to be similar if there exists an invertible operator $X:\hl_1\to \hl_2$ such that 
$XA_i=B_iX$ for all $i=1,\ldots,d$. For a non-negative integer $n$, as before, $\mathcal P_n$ denote the space of homogeneous polynomials  in $d$ variables of degree $n$.

\begin{theorem}\label{thmsim}
Let $\Omega\subseteq \mathbb C^d$ be a bounded  domain, and let $\hl_1$ and $\hl_2$ be two reproducing kernel Hilbert spaces determined by two positive definite kernels $K_1$ and $K_2$ respectively. 
Suppose that 
\begin{itemize}
\item[{\rm (i)}] the space of polynomials  $\mathcal P$ is dense in both $\hl_1$ and $\hl_2$,
\item[{\rm (ii)}] $\mathcal P_n$ is orthogonal to $\mathcal P_m$ if $m\neq n$,
\item[{\rm (iii)}]for each $i=1,2,$ the  $d$-tuple $\boldsymbol{M}^{(i)}=(M_1^{(i)},\ldots, M_{d}^{(i)})$ of multiplication operators by the coordinate functions $z_1,\ldots ,z_d$  on $\hl_i$ is bounded.
\end{itemize}
Then the following statements are equivalent.
\begin{itemize}
\item[\rm (i)]$\boldsymbol{M}^{(1)}$ and $\boldsymbol{M}^{(2)}$ are similar.
\item[\rm (ii)] 
There exist constants $\alpha, \beta >0$ such that  
\begin{equation}\label{eqnsim4}
\alpha \|p\|_{\hl_1}\leq  \|p\|_{\hl_2}\leq \beta  \|p\|_{\hl_1}, \qquad p \in \mathcal P.
\end{equation}
\item[\rm (iii)]$\hl_1=\hl_2.$
\item[\rm(iv)]There exists  constants $\alpha, \beta >0$ such that 
$$\alpha K_1 \preceq K_2 \preceq\beta K_1.$$
 \end{itemize}
\end{theorem}

\begin{proof}
The equivalence of (iii) and (iv) follows from the standard theory of reproducing kernel Hilbert spaces (cf. \cite{Aro}, \cite{PaulsenRaghupati}). Also it is clear that
\rm (ii) implies \rm (iii), since the polynomials are dense in both $\hl_1$ and $\hl_2$. If $\hl_1=\hl_2$, then the identity operator 
from $\hl_1$ to $\hl_2$ is a bounded invertible operator which intertwines the multiplication $d$-tuples $\boldsymbol M^{(1)}$  and $\boldsymbol M^{(2)}$, and consequently, \rm(iii) implies \rm(i). Now, to complete the proof, it remains to show that \rm (i) implies \rm (ii).

Suppose that $\boldsymbol{M}^{(1)}$ and $\boldsymbol{M}^{(2)}$ are similar. Then there exists an invertible operator $X:\hl_1\to \hl_2$ such that 
\begin{equation}\label{eqnsim}
XM_j^{(1)}=M_j^{(2)}X, \qquad j=1,\ldots,d.
\end{equation}

Since the subspaces $\mathcal P_n$, $n\geq 0$, are mutually orthogonal , it suffices to show that \eqref{eqnsim4} is satisfied for all $p\in \mathcal P_n$ and for some $\alpha, \beta >0$ (which is independent of $n$). Fix a polynomial $p$ in $\mathcal P_n$. Clearly, it follows from 
\eqref{eqnsim} that 
\begin{equation}\label{eqnsimp}
XM_{p}^{(1)}=M_{p}^{(2)}X.
\end{equation}
Let $\big(X_{r,s}\big)_{r,s=0}^\infty$  be the matrix representation of $X$   with respect to $\oplus \mathcal P_n$, that is, $X_{r,s}=P_{\mathcal P_r}X_{|\mathcal P_s}$. Similarly, let  $M_p^{(i)}=\big((M_p^{(i)})_{r,s}\big)_{r,s=0}^\infty$ be the matrix representation of $M_p^{(i)}$, $i=1,2$. Since $M_p^{(i)}$ maps $\mathcal P_s$ into $\mathcal P_{s+n}$, $i=1,2$, it clear that 
\begin{equation}
(M_p^{(i)})_{r,s}=\begin{cases}
(M_p^{(i)})_{|\mathcal P_s}, & \text{if $r=s+n$}\\
0, & \text{otherwise.}
\end{cases}
\end{equation}
Therefore it follows from \eqref{eqnsimp} that 
\begin{equation}\label{eqnsim3}
X_{r,s+n}(M_p^{(1)})_{s+n,s}=\begin{cases}
(M_p^{(2)})_{r,r-n}X_{r-n,s}, & \text{if $r-n\geq 0$}\\
0, &\text{otherwise.}
\end{cases}
\end{equation}
Choosing $r=n$ and $s=0$, we see that
\begin{equation}
(M_p^{(2)})_{n,0}X_{0,0}=X_{n,n}(M_p^{(1)})_{n,0}.
\end{equation}
Therefore
\begin{equation}\label{eqnsim2}
(M_p^{(1)})_{n,0}^*X_{n,n}^*X_{n,n}(M_p^{(1)})_{n,0}=
X_{0,0}^*(M_p^{(2)})_{n,0}^*(M_p^{(2)})_{n,0}X_{0,0}.
\end{equation}
Since $\|X_{n,n}\|\leq \|X\|$, we have 
$$X_{n,n}^*X_{n,n}\preceq \|X\|^2I.$$
Hence from \eqref{eqnsim2} we obtain
\begin{equation}\label{eqnsim5}
X_{0,0}^*(M_p^{(2)})_{n,0}^*(M_p^{(2)})_{n,0}X_{0,0}\preceq \|X\|^2 (M_p^{(1)})_{n,0}^*(M_p^{(1)})_{n,0}.
\end{equation}
Note that $X_{0,0}$ is a linear map from $\mathcal P_0$ to  $\mathcal P_0$ and $\dim \mathcal P_0=1$. Hence $X_{0,0}1=\eta 1$ for some $\eta\in \mathbb C$. Also, taking $p$ to be the polynomial $z_j$, $1\leq j\leq d,$ and $r=0$ in \eqref{eqnsim3} we see that
$$X_{0,s+1}(M_{z_j}^{(1)})_{s+1,s}=0, \qquad \text{for all}~~ s\geq 0.$$
Since this is true for all $j=1,\ldots,d$, it follows that 
$X_{0,s+1}=0$ for all $s\geq 0$. Moreover, since $X$ is invertible we must have $X_{0,0}\neq 0$. Otherwise, if $X_{0,s}=0$ for all $s$, then it is easy to see that $\mathcal P_0$ is orthogonal to range of $X$, which is a contradiction. Hence $X_{0,0}\neq 0$, and consequently $\eta\neq 0$. Therefore
\eqref{eqnsim5} gives 
\begin{equation*}
\langle (M_p^{(2)})_{n,0}X_{0,0}1, (M_p^{(2)})_{n,0} X_{0,0}1\rangle \leq \|X\|^2 \langle(M_p^{(1)})_{n,0}1, (M_p^{(1)})_{n,0}1 \rangle.
\end{equation*}
Consequently, 
\begin{equation*}
 |\eta|^2 \|p\|_{\hl_2}^2\leq \|X\|^2 \|p\|_{\hl_1}^2.
\end{equation*}
To finish the proof, note that \eqref{eqnsim} implies
\begin{equation}
X^{-1}M_j^{(2)}=M_j^{(1)}X^{-1},~~j=1,\ldots,d.
\end{equation}
Hence following the arguments used in the first part of this proof  we obtain that
\begin{equation*}
 |\zeta|^2 \|p\|_{\hl_1}^2\leq \|X^{-1}\|^2 \|p\|_{\hl_2}^2,
\end{equation*}
where $(X^{-1})_{0,0}1=\zeta. 1, \zeta\neq 0$.
This completes the proof.
\end{proof}

\begin{remark}
In the proof given above, we have shown that $X_{0,s}=0$ for all $s>0$. But using \eqref{eqnsim3}, it can be easily verified that $X_{r,s}=0$ for all $s>r$, that is, $X$ is lower triangular with respect to the decomposition $\oplus \mathcal P_n$. Consequently, $\zeta=\frac{1}{\eta}$.
\end{remark}

\begin{theorem}\label{thmsimilarity}
Let $\boldsymbol T_1$ and $\boldsymbol T_2$ be two operator tuples in $\mathcal A\mathbb{K}(\Omega)$. Suppose that $T_1\sim_u \boldsymbol{M}^{(a)}$ and $T_2\sim_u  \boldsymbol{M}^{(b)}$. Then the following statements are equivalent. 
\begin{itemize}
\item[\rm (i)]$\boldsymbol T_1$ and $\boldsymbol T_2$ are similar.
\item[\rm (ii)] 
There exist constants $\alpha, \beta >0$ such that  
\begin{equation}
\alpha \|p\|_{\hl^{(a)}}\leq  \|p\|_{\hl^{(b)}}\leq \beta  \|p\|_{\hl^{(a)}}, \qquad  p \in \mathcal P.
\end{equation}
\item[\rm (iii)]$\hl^{(a)}=\hl^{(b)}.$
\item[\rm(iv)]There exist  constants $\alpha, \beta >0$ such that 
$$\alpha K^{(a)} \preceq K^{(b)} \preceq\beta K^{(a)}.$$
\item[(v)] there exist constants
$\alpha, \beta>0$ such that 
\begin{equation*}
\alpha a_{\s}\leq b_{\s}\leq \beta a_{\s}, \qquad \s \in \overrightarrow  {\mathbb N}^r_0.
\end{equation*}
\end{itemize}
\end{theorem}
\begin{proof}
The equivalence of \rm(i), \rm(ii), \rm(iii) and \rm(iv) follows easily from Theorem \ref{thmsim}. Assume that
\rm (ii) holds. Then \rm (v) is easily verified by  
choosing any polynomial $p$ in $\mathcal{P}_{\s}$ and using $\|p\|^2_{\hl^{(a)}}=\frac{\|p\|^2_{\mathcal F}}{a_{\s}}$ and 
$\|p\|^2_{\hl^{(b)}}=\frac{\|p\|^2_{\mathcal F}}{b_{\s}}$ in  \eqref{eqnsim4}. Also,  it is trivial to see that \rm (v) implies
\rm (iv).
\end{proof}

\begin{corollary}
Let $\nu_1, \nu_2 > \frac{a}{2}(r-1).$ Then the  $d$-tuple of multiplication operators $\boldsymbol{M}^{(\nu_1)}$ on 
$\mathcal H^{(\nu_1)}$ and $\boldsymbol{M}^{(\nu_2)}$ on  $\mathcal H^{(\nu_2)}$ are similar if and only if $\nu_1=\nu_2.$
\end{corollary}
\begin{proof}
Suppose that $\boldsymbol{M}^{(\nu_1)}$ and $\boldsymbol{M}^{(\nu_2)}$  are similar. Then, by Theorem \ref{thmsimilarity}, there exist constants $\alpha, \beta >0$ such that $\alpha (\nu_1)_{\s} \leq (\nu_2)_{\s}\leq \beta (\nu_1)_{\s}$ for all  $\s\in \overrightarrow  {\mathbb N}^r_0$. Take $\s=(t,0,\ldots,0)$, $t\in \mathbb N_0.$ By the properties of the Gamma function we have  $\frac{(\nu_1)_{\s}}{(\nu_2)_{\s}}=\frac{(\nu_1)_t}{(\nu_2)_t} \sim t^{\nu_1-\nu_2}.$ Hence $\nu_1=\nu_2$. The other implication is trivial.
\end{proof}

\medskip \textit{Acknowledgment}.
The authors have benefited greatly from many hours of discussion on the topic of this paper with Prof. Harald Upmeier during his visits to the Department of Mathematics, Indian Institute of Science, Bangalore as the InfoSys Chair Professor. We also express our sincere thanks to Prof. Gadadhar Misra for several fruitful discussions and many useful suggestions, which resulted in a considerable refinement of the original draft.


\begin{thebibliography}{HD}
\bibitem{Arazy} 
J. Arazy, \emph{A survey of invariant {H}ilbert spaces of analytic functions on bounded symmetric domains}, in: Contemporary Mathematics, vol. 185  (1995), 7-65.

\bibitem{Zhang}
J. Arazy, G. Zhang, \emph{Homogeneous multiplication operators on bounded symmetric domains},
J. Funct. Anal. {\bf 202} (2003), 44-66.

\bibitem{Aro}
N. Aronszajn, \emph{Theory of reproducing kernels}, Trans. Amer. Math. Soc.
  \textbf{68} (1950), 337-404. 

\bibitem{AHDH}
W. Arveson, D. W. Hadwin, T. B. Hoover and E. E. Kymala,
\emph{Circular operators}, Indiana Univ. Math. J. {\bf 33} (1984), 583-595.

\bibitem{At}
A. Athavale, 
\emph{A note on Cartan isometries},
 New York J. Math. {\bf 25} (2019), 934–948.
 
 
\bibitem{MisraBagchi}
B. Bagchi and G. Misra,
\emph{ Homogeneous tuples of multiplication operators on twisted
{B}ergman spaces},
J. Funct. Anal. {\bf 136} (1996).

\bibitem{Ca}
\'E. Cartan, \emph { Sur les domaines born\'es homog\'enes de l'espace den variables complexes}, Abh. Math. Sem. Univ. Hamburg {\bf 11} (1935), 116-162.

 \bibitem{CY}
 S. Chavan and D. Yakubovich,
 \emph{ Spherical tuples of Hilbert space operators},
 Indiana Univ. Math. J. {\bf 64} (2015), 577-612.
 
 
\bibitem{CD}
M. J. Cowen and R. G. Douglas, \emph{Complex geometry and operator theory},
  Acta Math. \textbf{141} (1978), 187-261. 

\bibitem{CDopen}
\bysame, \emph{Operators possessing an open set of eigenvalues}, Functions,
  series, operators, {V}ol. {I}, {II} ({B}udapest, 1980), Colloq. Math. Soc.
  J\'anos Bolyai, vol. 35, North-Holland, Amsterdam, (1983), 323-341.
   

\bibitem{Curtosalinas}
R. E. Curto and N. Salinas, \emph{Generalized bergman kernels and the
  cowen-douglas theory}, Amer. J. Math. \textbf{106} (1984), 447-488.

\bibitem{equivofquotient}
R. G. Douglas and G. Misra, \emph{Equivalence of quotient {H}ilbert modules},
  Proc. Indian Acad. Sci. Math. Sci. \textbf{113} (2003), 281-291.





\bibitem{EPut} 
J. Eschmeier and M. Putinar, \emph{Spectral Decompositions and Analytic Sheaves}, Oxford University Press, 1989.

\bibitem{Wallachset}
J. Faraut and A. Kor\'anyi, \emph{ Function spaces and reproducing kernels on
  bounded symmetric domains}, J. Funct. Anal. \textbf{88} (1990), 64-89. 

\bibitem{G}
R. Gellar, \emph{Circularly symmetric normal and subnormal operators}, J. Analyse Math. {\bf 32 }(1977),  93-117. 

\bibitem{Hua}
L. K. Hua, \emph{Harmonic analysis of functions of several complex variables in the classical domains}, Translations of Mathematical Monographs, Vol. 6,  American Mathematical Society, Providence, R.I., 1963.
	
	
	
	









%


 



  
%
%
%
%
%
%
%
%
%
%
%
%
%




%

  \bibitem{MK}
  A. Kor\'anyi and G. Misra,, \emph{A classification of homogeneous operators in the {C}owen-{D}ouglas class}
 Adv. Math. \textbf{226} (2011), 5338-5360.
  
\bibitem{MK1}
\bysame, \emph{Homogeneous Hermitian holomorphic vector bundles and the {C}owen-{D}ouglas class over bounded symmetric domains}
Adv. Math. \textbf{351} (2019), 1105-1138.
  
\bibitem{Ku}
S. Kumar, \emph{Spherically balanced Hilbert spaces of formal power series in several variables–II}, Complex Anal. Oper. Theory \textbf{10} (2016), 505-526.


\bibitem{Loos}
O. Loos, \emph{Bounded symmetric domains and jordan pairs}, University of California, Irvine, 1977.


\bibitem{Misra}
G. Misra, \emph{Curvature and the backward shift operators}, Proc. Amer. Math.
  Soc. \textbf{91} (1984),  105-107.
  
  \bibitem{MS} 
   G. Misra, N. S. N. Sastry, \emph{Homogeneous tuples of operators and representations of some classical groups}, J. Operator Theory {\bf 24} (1990), 23-32.
   
  \bibitem{MU}
  G. Misra, H. Upmeier, \emph{Homogeneous vector bundles and intertwining operators for symmetric domains}, Adv. Math. {\bf 303} (2016), 1077-1121.
  
  \bibitem{Mok}
  N. Mok, \emph{ Metric rigidity theorems on hermitian localy symetric manifold},
     Series in Pure Mathematics, vol. {6},  World Scientific, 1989.
   
  \bibitem{RN}
  R. Narasimhan, \emph{Several complex variables}, Chicago Lect. Math., The University of Chicago
  Press, 1971
  
  \bibitem{PaulsenRaghupati}
  V. I. Paulsen and M. Raghupathi, \emph{An introduction to the theory of
  	reproducing kernel {H}ilbert spaces}, Cambridge Studies in Advanced
  Mathematics, vol. 152, Cambridge University Press, Cambridge, 2016.
  
  
  \bibitem{Sh} 
  A. Shields, \emph{ Weighted shift operators and analytic function
  	theory, in Topics in Operator Theory}, Math. Surveys
  Monographs, vol. 13, Amer. math. Soc., Providence, R.I., 1974, 49-128.



\bibitem{Hu}

H. Upmeier, \emph{Jordan algebras and harmonic analysis on symmetric spaces}, Amer. J. Math. {\bf108} (1986), 1-25.

\bibitem{upmeier}
 \bysame, \emph{ Toeplitz operators and index theory in several complex variables},  Operator Theory: Advances and Applications, vol. 81, Birkhäuser Verlag, Basel, 1996.

\bibitem{Up}
 \bysame, \emph{ Eigenvalues for $K$-invariant Toeplitz operators}, preprint. 

\bibitem{Wilkins}
D. R. Wilkins, \emph{Homogeneous vector bundles and {C}owen-{D}ouglas operators}, 
Internat. J. Math.
\textbf{4} (1993), 503-520.


%
%
%
%
%
%
%
%
%
%
%
%
%
%


\end{thebibliography}
\end{document}